\documentclass{amsart}
\usepackage[T1]{fontenc}
\usepackage[utf8]{inputenc}

\usepackage{amsmath,amssymb,amsfonts,amsthm}
\usepackage{graphicx}
\usepackage{xcolor}

\definecolor{orcid-green}{RGB}{166, 206, 57}

\usepackage{multicol}
\usepackage{float}
\usepackage{xcolor}
\usepackage{url}
\usepackage{bm}

\usepackage{tikz}
\usepackage{pifont}

\newtheorem{theorem}{Theorem}[section]
\newtheorem{lemma}[theorem]{Lemma}
\newtheorem{proposition}[theorem]{Proposition}
\newtheorem{corollary}[theorem]{Corollary}

\newtheorem{definition}[theorem]{Definition}
\newtheorem{remark}[theorem]{Remark}
\theoremstyle{definition}

\newtheorem{example}[theorem]{Example}
\newtheorem{problem}[theorem]{Problem}

\newcommand{\N}{\mathbb N}

\newcommand{\R}{\mathbb R}
\newcommand{\ve}{\varepsilon}
\newcommand{\p}{\mathcal{P}}
\newcommand{\A}{\mathcal{A}}

\definecolor{xlinkcolor}{cmyk}{1,0.6,0,0}

\usepackage[
  bookmarks=false,
  pdfnewwindow=true,
  colorlinks=true,
  linkcolor=xlinkcolor,
  citecolor=xlinkcolor,
  filecolor=xlinkcolor,
  urlcolor=xlinkcolor,
  final=true
]{hyperref}

\makeatletter
\@namedef{subjclassname@2020}{%
  \textup{2020} Mathematics Subject Classification}
\makeatother

\begin{document}


\title[Spectre operator, achievement sets and sets of P-sums]{Spectre operator, achievement sets and sets of P-sums in a hyperspace of compact sets}

\author{Piotr Nowakowski}
\address{Faculty of Mathematics and Computer Science
\\University of Lodz
\\Banacha 22,
90-238 \L\'{o}d\'{z}
\\Poland\\
 ORCID: 0000-0002-3655-4991}
\email{piotr.nowakowski@wmii.uni.lodz.pl}

\author{Franciszek Prus-Wiśniowski}
\address{Instytut Matematyki\\
Uniwersytet Szczeci\'{n}ski\\
ul. Wielkopolska 15\\
70-453 Szczecin\\
Poland\\
 ORCID: 0000-0002-0275-6122}
\email{franciszek.prus-wisniowski@usz.edu.pl}

\author{{Filip Turobo\'s}}
\address{{
Institute of Mathematics, Lodz University of Technology,
al. Politechniki 10, 93-590 \L \'{o}d\'{z}, Poland
\\ ORCID: 0000-0002-5782-6159}}
\email{{filip.turobos@p.lodz.pl}}

\begin{abstract}
 Let $(X,+,d)$ be an Abelian metric group and $A\subset X$. We investigate the spectre of a set $A$, defined as the set of all elements $z\in X$ such that for every $x\in A$ either $x+z \in A$ or $x-z \in A$. We consider the corresponding to this notion operator $S$ acting on the hyperspace of compact sets and examine its properties. Furthermore, we study the families of achievement sets and sets of $P$-sums in this hyperspace, as well as prove some properties of achievement sets in the plane.

\end{abstract}

\maketitle

\section{INTRODUCTION}

The center of distances is quite a new and intriguing notion in the setting of metric spaces. It stems from the paper of Bielas, Plewik and Walczy\'nska \cite{Bielas}, where it was introduced as follows:

\begin{definition}[Center of distances]
Let $(X,d)$ be a metric space. The \textbf{center of distances} for a metric space $X$ is a subset $C(X)\subset [0,+\infty)$ defined
as
\[
C(X):=\{ \alpha : \forall_{x\in X} \; \exists_{y\in X}\,\, d(x,y)=\alpha \}.
\]
\end{definition}
Of course, we can apply the said concept to a subset $A$ of a metric space $X$, as the distance function is naturally inherited by the subspace.
Given a point $x\in X$ and a subset $A\subset X$, we can consider the set of distances between $x$ and the points of $A$ defined by $D_x(A):=\{d(x,y):\ y\in A\,\}$. Then the set of distances of $A$ can be described by $D(A)=\bigcup_{x\in A}D_x(A)$. Moreover, the center of distances of $A$ is $C(A)=\bigcap_{x\in A}D_x(A)$.

This notion was also subsequently used in the papers of Bartoszewicz, Głąb et al. \cite{AB1,ABart2} for various purposes, which are outside of the scope of this article.
The exploration of similar concept in group setting was done by Kula and Nowakowski in \cite{Piotrus}, where they have introduced a concept known as the \textit{spectre} of a set. This generalization of the center of distances has been defined as follows:

\begin{definition}[Spectre of a set]
    Let $(X,+)$ be an Abelian group and $A\subset X$. We define the spectre of a set $A$ (denoted by $S(A)$) as:
\[
S(A):= \{z \in X \colon \forall_{x\in A} \, x+z\in A \, \lor \, x-z\in A\} .
\]
\end{definition}

On the real line, the notions of spectre and the center of distances share many common properties. The former concept, however, can be applied in a more effective manner in more algebra-oriented settings as well.

Before proceeding any further, let us point out an important distinction between the center of distances and the spectre of a set. In the first case, regardless of the structure of the elements of the metric space at hand, the center is always a subset of non-negative reals. This poses certain difficulties, as the operator assigning to a subspace its center of distances cannot be iterated when the metric space is not embedded in real line.

In contrast, for the spectre of a set $A\subset X$, $S(A)$ is itself a subset of $X$. This property is particularly convenient, especially when dealing with high-dimensional metric spaces, as demonstrated in the original paper which introduced this notion \cite{Piotrus}.

Having established differences between these notions, it is worth pointing out two additional similarities between them. When considering operators mapping a given set (in the appropriate setting) either its center of distances or spectre of that set, they always contain the "\textit{neutral element}" of sorts: zero in the first case, and the group neutral element in the second scenario. In fact, in real settings these notions coincide to some extent, but we will justify this claim in the latter part of the article.

\section{Spectre operator behaviour}

Throughout this paper we consider a metric space $(X,d)$ equipped with an additional algebraic structure, therefore we also regard $X$ as an Abelian metric group $(X,+,d)$. We also need to assume that the metric $d$ is translation-invariant, i.e.,
\[
\forall_{x,y,z\in X} \qquad 
d(x,y) = d(x+z,y+z).
\]

We denote the neutral element of $+$ operation by $0$ and the cardinality of a given set $A$ is denoted by $|A|$. The open ball with center at $x$ and radius $r$ is denoted by $B(x,r)$ and the closed ball by $\bar{B}(x,r)$. We also introduce the following two families of subsets of $X$:
\begin{itemize}
\item $K(X)$, which we denote by $K$ if the setting is clear, which stands for all compact non-empty subsets of $X$;
\item $K_0(X)$ or simply $K_0$, which denotes the family of all compact subsets of $X$ containing $0$.
\end{itemize}

From the very definition of the spectre of a set arises a list of its properties, which is formally encompassed in the following observation from \cite{Piotrus}:

\begin{remark}\label{podst}
For any non-empty set $A\subset X$ we have:
\begin{enumerate}
    \item $0\in S(A)$;
    \item if $0\in A$ then $S(A)\subset A\cup (-A)$;
    \item for any $z\in X$ we have $S(A) = S(A+z)$;
    \item for any $z\in X$ we have
    \[
    z\in S(A) \iff (-z)\in S(A).
    \]
\end{enumerate}
\end{remark}
We now proceed with another easy but nonetheless useful lemma.
\begin{lemma} \label{ogr}
Assume that $A \subset X$, and for some $x \in X$ and $r>0$, $A \subset B(x,r)$. Then we have $S(A) \subset B(0,2r)$.
\end{lemma}
\begin{proof}
Let $y \in A$. Then $d(x,y) < r$.
From Remark \ref{podst} (2) and (3), the assumption and the fact that $d$ is translation invariant we obtain
$$S(A) = S(A-y) \subset (A-y) \cup (-A + y) \subset B(x-y,r) \cup B(y-x,r) \subset B(0,2r).$$
\end{proof}

In the following part, we will also need the following results, which are taken from \cite{Piotrus}.
\begin{proposition}\label{capcup}
For any family $(A_i)_{i\in I}$ we have $$\bigcap_{i \in I} S(A_i) \subset S(\bigcup_{i \in I} A_i).$$

\end{proposition}

\begin{lemma}\label{Lemma:Kompakt} The spectre of a non-empty compact set $A$ is compact.
\end{lemma}

Thanks to Lemma \ref{Lemma:Kompakt}, we can define an operator $S:K(X)\to K_0(X)$, assigning to each compact subset of $X$ its spectre. This allows us to explore the properties of such a mapping.

The following lemma reveals an interesting property when comparing the notions of the spectre and the center of distances. In particular, it shows that in this more general setting we are able to easily provide the sufficient condition for a set to be invariant with respect to the operation of taking the spectre of a set.

\begin{lemma}\label{lemma:subgrup}
    Let $Z$ be a subgroup of $(X,+)$. Then $S(Z)=Z$.
\end{lemma}
\begin{proof}
    Clearly, for any $x,y \in Z$ we have $x+y\in Z$, which justifies the fact that $Z\subset S(Z)$. Now, let $y\in S(Z)$. Then $x'=x+y\in Z$ or $x'' = x-y \in Z$. In the first case, $y=x'-x \in Z$ and in the second case $y=x-x'' \in Z$, which proves the reverse inclusion.
\end{proof}

Unfortunately, this implication cannot be reversed. The counterexample is surprisingly simple.

\begin{example}
    Consider a subset $A=\{-1,0,1\}$ of $(\mathbb{R},+)$. Then $A$ is certainly not a subgroup, however $S(A)=A$.
\end{example}

It turns out that $S$ neither needs to be continuous, nor does it need to be surjective. This can be observed in the following lemmas, which provide a method of obtaining such counterexamples. Let us start by considering the lack of the surjectivity:

\begin{lemma}
    Let $(X,+)$ be an Abelian group endowed with a translation-invariant metric $d$. Additionally, assume that there exist $x,y\in X$ such that $x+y\neq 0$ and $\max\{\operatorname{rank}(x),\operatorname{rank}(y) \}\geqslant 3$. Then, the spectre operator $S:K(X)\to K_0(X)$ is not surjective.
\end{lemma}
\begin{proof}
Consider subset $A_{x,y}=\{0,x,y\}\in K_0(X)$. It is clearly compact (as it is a union of three singletons), however it cannot be the spectre of any set. Indeed, if $A_{x,y}$ were a spectre of some set $B$, then $-x$ and $-y$ would both have to belong to $A_{x,y}$. Clearly, $-x\neq y$ and $-y\neq x$, whereas the fact that at least for one of these elements has rank greater than $2$ implies that either $x\neq -x$ or $y\neq -y$. Therefore, $A_{x,y}$ cannot be the spectre of any set.
\end{proof}

Having established that the spectre operator $S$ fails to be surjective in general, it is natural to ask under what additional assumptions a given set may arise as a spectre.

\begin{problem}
Let $A \subset X$ be such that $x\in A \Rightarrow -x \in A$ and $0 \in A$. Does there exist a set $B\subset X$ such that $S(B)=A$? If the answer to that question is positive, can we additionally say something about $B$, for example when $A$ is compact? Can we draw further conclusions about the nature of the set $B$ in the case where $X=\R^2$?
\end{problem}

A similar question for center of distances has the affirmative answer which was recently shown by Mateusz Kula (see \cite{Ku}).
In fact, we can repeat his proof with only minor modifications (specifically, it suffices to replace $\R$ and $[0,\infty)$ by $X$) to obtain the following:
\begin{theorem}
Let $X$ be a linear space over $\R$ such that a Hamel basis of $X$ over $Q$ has the same cardinality as $X$. Let $A \subset X$ be such that $x\in A \Rightarrow -x \in A$ and $0 \in A$. Then there exists $B \subset X$ such that $S(B)=A$.
\end{theorem}

\section{On the continuity of the spectre operator}

We now shift our attention to the issue of continuity of $S$ understood as an operator from $\mathcal{P}(X)$ to $\mathcal{P}(X)$. The question of the continuity of this operator was raised by Paolo Leonetti during XXXVII International Summer Conference on Real Functions Theory 2023 in Rowy.

We assume that the Reader is well acquainted with the concept of Pompeiu-Hausdorff metric\footnote{Which is approaching its $120^{th}$ anniversary \cite{hausdorfik}.}, nevertheless, we recall it's definition below.

\begin{definition}[Pompeiu-Hausdorff distance]
For two subsets $A,B$ of a metric space $X$, we define the Pompeiu-Hausdorff distance between them as
\[
d_H(A,B) = \max \left\{ \sup_{x\in A} d(x,B), \sup_{y\in B} d(A,y) \right\},
\]
where for $y\in X$ and $A\subset X$, $d(A,y)$ is defined as $d(A,y) = \inf_{a\in A} d(a,y)$.
\end{definition}

By $G(x)$ we denote the subgroup of $X$ generated by an element $x\in X$.

\begin{lemma}
Let $(X,+)$ be an Abelian group endowed with a translation-invariant metric $d$. Assume that there exists a non-neutral element $x\in X$ such that $G(x)\neq X$ and a sequence $(x_n)\subset X\setminus G(x)$ such that $x_n$ converges to $x$. Then, the spectre operator $S:K(X)\to K_0(X)$ is not continuous.
\end{lemma}
\begin{proof}
    Let $A = G(x)$ and $B_n = G(x) \cup \{x_n\}$ for all $n\in \mathbb{N}$. Note that for sufficiently large $n$ we have $d_H(A,B_n)=d(x,x_n)$, which tends to $0$ as $n$ tends to infinity. Since $A$ is a subgroup of $X$, Lemma \ref{lemma:subgrup} implies that $S(A)=A$. Moreover, $z\notin S(B_n)$ for all $z\in G(x) \setminus\{0\}$, since otherwise for some $y\in G(x)$ we would have $y=z+x_n$, which in turn would imply $x_n=y-z\in G(x)$, a contradiction with the assumptions on the sequence $(x_n)$. On the other hand, by Lemma \ref{podst} (2), $S(B_n) \subset \{0\} \cup B_n \cup -B_n$, so the only other points which can be in $S(B_n)$ are $x_n$ and $-x_n$, but this is also impossible by a similar argument as before. Hence $S(B_n) = \{0\}$. Therefore,
    \[d_H(S(A),S(B_n))\geqslant d(0,x)>0.\]
\end{proof}

The simplest example illustrating this lemma in practice can be found on the real line:
\begin{example}
Consider the metric group $(\mathbb{R},+,d)$, with $x:=1$ and $x_n:=1-\frac{1}{n}$ for $n\in\mathbb{N}$. Then $G(x)=\mathbb{Z}\neq \mathbb{R}$, hence the assumptions of the theorem are satisfied. One can easily see that \[S(G(x))=S(\mathbb{Z}) = \mathbb{Z},\]
while simultaneously $S(G(x)\cup \{1-\frac{1}{n}\})=\{0\}$ for all $n\in\mathbb{N}$. Clearly, there is no convergence $S(G(x)\cup \{1-\frac{1}{n}\})\to S(G(x))$ in this case.
\end{example}

For subsequent investigations we will require a slightly more general version of the concept of a non-sliding set, as formulated in \cite{Bartur}. However, before introducing the generalized definition, let us give a constructive proof of the existence of uncountable non-sliding sets in $[0,\,1]$.

A set $A\subset\mathbb R$ is said to be non-sliding if for every non-zero $t\in \mathbb R$ the intersection $A\cap(A+t)$ consists of at most one point. A non-sliding set $A\subset\mathbb R$ is one in which each positive distance between points of $A$ is realized by exactly one two-element subset of $A$.
Under this definition, all sets containing at most two points (including the empty set) are non-sliding. It is easy to see that for any $n\in\mathbb N$ the family of all non-sliding  subsets of $[0,\,1]$ consisting of exactly $n$ points has the cardinality of the continuum. The center of distances of any non-sliding set is $\{0\}$, except for all two-element sets $A=\{a,\,b\}$, because then $C(A)=\{0,\,|a-b|\}$.

For $q\in(0,\,\frac12)$, let $A_q$ denote the set consisting of $0$ and all terms of a geometric sequence $(q^n)_{n\in\mathbb N_0}$. This forms an example of a countably infinite, closed, non-sliding subset of $[0,\,1]$, since $D(A_q)\setminus\{0\}=\{q^n-q^k:\ n,k\in\mathbb N_0,\,k>n\,\}\cup\{q^n:\ n\in\mathbb N\}$ and each of these distances is realized in $A_q$ by a unique two-point subset.

With the help of transfinite induction, that is, under the axiom of choice, it is possible to prove existence of uncountable non-sliding subsets of $[0,\,1]$. Before we do it, we have to introduce an additional notation and one
auxiliary result.
Let $(W,\,\le)$ be a well-ordered set. For every $u\in W$ the set $W(u):= \{w\in W \colon w < u\}$ will be called an initial segment of $W$.
We will use the following theorem from \cite{S}.

\begin{theorem}\label{pom}
Let $(W,\,\le)$ be a well-ordered set, $X$ a set, and $\mathcal{F}$ the set of all maps with domain an initial segment of $W$ and range contained in $X$. If $G:\,\mathcal{F}\to X$ is any map, then there is a unique map $f:\,W\to X$ such that for every $u\in W$
$$
f(u)\ = G\bigl(f\big|_{W(u)}\bigr).
$$
\end{theorem}

We are ready now for the constructive proof of existence of non-sliding sets of cardinality continuum.

\begin{proposition}
\label{padd}
There exists an uncountable non-sliding subset of $[0,\,1]$.
\end{proposition}
\begin{proof}

Our first claim is that for every countable non-sliding subset $A$ of $[0,\,1]$ there is a non-sliding superset $B\subset[0,\,1]$ such that $B\setminus A$ consist of exactly one point. Indeed, since the set $D(A)$ is countable, so is
$$
\tilde{A}\ :=\
\bigcup_{a\in A}((a+
D(A))\cup (a-D(A)))
$$
(we agree to write $\tilde{\emptyset}:=\emptyset$). By choosing any point $x\in[0,\,1]\setminus\tilde{A}$ and defining $B:=A\cup\{x\}$, we obtain the desired set.

Let $\omega_1$ be the smallest uncountable ordinal. We claim that there is a strictly increasing transfinite sequence $(A_\xi)_{\xi<\omega_1}$ of non-sliding subsets of $[0,\,1]$. Indeed, let $(W,\,\le)$ be the well-ordered set of all countable ordinals, let $X$ be the set of all countable and non-sliding subsets of $[0,\,1]$ and let $\mathcal{F}$ be the set of all maps with domain being some initial segment of $W$ and with range contained in $X$. By our first claim, there is a map $g:\,X\to X$ such that $A\subset g(A)$ and $|g(A)\setminus A|=1$ for all $A\in X$. Define a map $G:\,\mathcal{F}\to X$ by setting
$$
G(\phi)\ :=\ g\left(\bigcup_{\xi<\eta}\phi(\xi)\right)
$$
if $\phi:\,\{\xi\in W:\ \xi<\eta\,\}\,\to\,X$ and the sequence $(\phi(\xi))_{\xi<\eta}$ is increasing and by setting $G(\phi):=\emptyset$ otherwise. $G$ is well defined because if $(\phi(\xi))_{\xi<\eta}$ is increasing, then $\bigcup_{\xi<\eta}\phi(\xi)) \in X$. By Theorem \ref{pom}, there is a unique map $f:\,W\to X$ such that for every $\xi\in W$
$$
f(\xi)\ = G\bigl(f\big|_{W(\xi)}\bigr).
$$
The map $g$ used in the definition of $G$ guarantees that the sequence $(f(\xi))_{\xi<\omega_1}$ is strictly increasing which completes the justification of our second claim.

The set $F:=\bigcup_{\xi<\omega_1}f(\xi)\,\subset\,[0,\,1]$ is non-sliding and uncountable.
\end{proof}

In general, the closure of a non-sliding set needs not to be non-sliding, as it is in the case of $A:=\{0\}\cup\{\frac12+\frac1{2^n}:\ n\in\mathbb N\,\}$. Thus, Proposition \ref{padd} does not answer the following open problem:
\begin{problem}
Do there exist uncountable closed non-sliding subsets of $[0,\,1]$?
\end{problem}

We now return to the main line of our note and formulate a concept generalizing the non-sliding sets.

\begin{definition}[Net-set]
    We will say that $A$ is a net-set in $(X,+)$ if it is finite, contains at least three elements and for every pair of distinct two-element subsets of $A$, say $\{z,t\}$ and $\{u,v\}$, the following property holds:
    \[\{z-t, t-z\} \cap \{v-u, u-v\} = \emptyset. \]
    The family of all net-sets in $X$ will be denoted by $NT(X)$.
\end{definition}
\begin{remark}
It is easy to see that in metric groups every non-sliding set is a net-set. However, the opposite inclusion does not hold. A simple example of this fact is a set $\{(0,1),(1,0),(0,0)\}$ in $\R^2$. It is easy to see that it is a net-set, but not a non-sliding set.

\end{remark}

The main motivation behind the definition of net-sets is to ensure that their spectres are trivial (i.e., $S(A)=\{0\}$ for any $A\in NT(X)$). Also notice that in the definition of a net-set we do not require the two-element subsets of $A$ to be disjoint.

The following results are analogues of theorems from \cite{Bartur} which discuss the properties of the centre of distances. Similar methods of proof are employed in our paper.

\begin{lemma}\label{Ntcont}
If $A\in NT(X)$, then the operator $S$ is continuous at $A$.
\end{lemma}
\begin{proof}
Fix $\varepsilon > 0$ and consider $A-A:=\{x-y : x,y\in A \}$. Since $A$ is finite, there exists the smallest positive distance between elements of $A-A$, which we denote by
$$\eta := \min \{ d(x,y) : \, x,y\in A-A, \, x\neq y \}.$$
Consider a compact set $B\in K(X)$ such that
\[d_H(A, B) < \delta := \min\{\frac{\eta}{4},\frac{\ve}{4} \}.
\]
Suppose that $ d_H(S(A), S(B)) \geq \varepsilon$, that is, $d_H(\{0\}, S(B)) \geq \varepsilon$.
Then there exists $x \in S(B)$ such that $x \neq 0$. Observe that $\eta \leq \min\{d(y,z)\colon y,z \in A, y\neq z\}$.
Indeed, since $d$ is translation-invariant, for any $y,z \in A$ we have $d(y,z)=d(z-y,0) \geq \eta$. Therefore, since $d_H(A, B) <\frac{\eta}4 $, we have $|B| \geq |A|$. Hence, we can find three distinct elements $b_1, b_2, b_3 \in B$ and an element $b_4 \in B$ (not necessarily distinct from the previous ones, as $B$ need not be a net-set) such that $x = b_1 - b_2 = b_3 - b_4$.

Choose $a_i \in A$ such that $d(a_i,b_i)<\delta$ for all $i$.
Then $c_1 := a_1 - a_2$ satisfies
\[d(c_1,x) = d(a_1-a_2,b_1-b_2) \leqslant d(a_1-a_2,b_1-a_2)+d(b_1-a_2,b_1-b_2) $$$$= d(a_1,b_1)+d(-a_2,-b_2) <2\delta = \frac{1}{2}\min\{\varepsilon,\eta\},\]
where the middle equality stems from the metric being translation-invariant. Same reasoning can be conducted for $c_2:= a_3 - a_4$, proving that $c_1,c_2\in B(x,\frac{\delta}{2})$. Thus, $d(c_1,c_2)<\eta$. However, since $A$ is a net-set, we have
\[
\{a_1-a_2,a_2-a_1\}\cap \{a_3-a_4,a_4-a_3\} = \{c_1,-c_1\}\cap \{c_2,-c_2\} = \emptyset.
\]
Since these two sets are disjoint, we have $c_1\neq c_2$. This shows that the distance $d(c_1,c_2)$ is positive and smaller than $\eta$. This contradicts the definition of $\eta$, which was defined as the smallest distance between elements of $A-A$.
\end{proof}
Since every compact set is totally bounded, the family of all finite subsets of $X$ is dense in $K(X)$. Using this fact we obtain the following results:
\begin{lemma} \label{gestosc}
    Let $(X,+,d)$ be a metric group with $d$ translation-invariant and such that every nonempty open ball in $X$ contains infinitely many elements. Then
    \begin{itemize}
    \item[(i)] the family $\{A\subset X\colon A \mbox{ is finite, } S(A) \neq \{0\}\}$ is dense in $K(X)$;

    \item[(ii)] the family $NT(X)$ is dense in $K(X)$.
    \end{itemize}
\end{lemma}
\begin{proof}
Let $B$ be a finite set in $X$ and fix $\ve>0$. To prove both assertions, it suffices to find a set $A$ in the appropriate family such that $d_H(A,B) < \ve$.

\textbf{Ad (i):} Let $x \in B(0,\ve)\setminus\{0\}$. Define $A:= B\cup (B+x).$ Then $S(A) \supset\{0,x\}$ and $d_H(A,B)< \ve.$

\textbf{Ad (ii):} Let $B=\{b_1,b_2,\dots, b_n\}$. If $n=1$, then it suffices to define $A = \{b_1, b_1+x, b_1+y\}$, where $x,y \in B(0,\ve)\setminus\{0\}$, $y\not\in \{x,2x,-x,-2x\}$ and $x\not\in \{-2y,2y\}$. Clearly, $A \in NT(X)$. If $n=2$, let $A=\{b_1,b_2, b_1+x\}$, where $x \in B(0,\ve)\setminus\{0,b_2-b_1,b_1-b_2\}$ and $b_1 \pm 2x \neq b_2$. Then $A \in NT(X)$ as well. Assume that $n \geq 3$. Set $a_1= b_1$, $a_2 = b_2$. Suppose that $a_i$ has already been defined for $i \leq k < n$ in such a way that $\{a_i\colon i \leq k\}$ is either a net-set or contains exactly two elements. If $\{a_i\colon i \leq k\} \cup\{b_{k+1}\}$ is a net-set, put $a_{k+1}=b_{k+1}$. Otherwise, there exists $x \in B(0,\ve)$ such that the set $\{a_i\colon i \leq k\} \cup\{b_{k+1}+x\}$ is a net-set (since only finitely many $x$ in $B(0,\ve)$ fail to satisfy this, and each ball is infinite). Then set $a_{k+1} = b_{k+1} +x$. Finally, define $A :=  \{a_i\colon i \leq n\}$, which is clearly a net-set.
\end{proof}

\begin{corollary} \label{imp1}
Let $(X,+,d)$ be a metric group with $d$ translation-invariant and such that every nonempty open ball in $X$ contains infinitely many elements. Let $A \in K(X)$. If $S$ is continuous at $A$, then $S(A) = \{0\}$.
\end{corollary}
\begin{proof}
The operator $S$ on $NT(X)$ is constantly equal to $\{0\}$, and by Lemma \ref{gestosc} the set $NT(X)$ is dense in $K(X)$. It follows that $S$ is constantly equal to $\{0\}$ on all its points of continuity.
\end{proof}
\begin{lemma}\label{conv}
Let $A\in K(X).$ Assume that $(A_n)$ is a sequence in $K(X)$ convergent to $A$ and the sequence $(S(A_n))$ is also convergent in $K(X)$ to some $B\in K(X)$. Then $B \subset S(A)$.
\end{lemma}
\begin{proof}
Let $b \in B$. We want to show that $b \in S(A)$.
 Take any $a \in A$ and choose sequences $(a_n)$ and $(b_n)$ in such a way that $a_n \in A_n$, $b_n \in S(A_n)$, $a_n\to a$ and $b_n\to b$.
 For every $n \in \N$, we have either $a_n+b_n \in A_n$ or $a_n-b_n \in A_n$.
 Without loss of generality, assume that for infinitely many $n$ we have $a_n+b_n \in A_n$; by passing to a subsequence, we may assume that it holds for all $n\in\mathbb{N}$.
 Due to compactness of $A$, as well as the convergence of $A_n\to A$, we obtain that the sequence $(a_n+b_n)$ has a subsequence convergent to some $s \in A$. On the other hand, $a_n+b_n \to a+b$, hence $a+b = s \in A$. As $a\in A$ was arbitrary, we conclude that $b \in S(A)$.
\end{proof}

\begin{corollary} \label{imp2}
Let $(X,+,d)$ be a complete locally compact metric group with a translation-invariant metric $d$. Let $A \in K(X)$. If $S(A) = \{0\}$, then $S$ is continuous at $A$.
\end{corollary}
\begin{proof}
Let $(A_n)$ be a sequence in $K(X)$ converging to $A$. Then the set $B:=\bigcup_{n\in\N} A_n \cup A$ is bounded, so there exist $x \in X$ and $r>0$ such that $B \subset B(x,r)$ and consequently $A_n \subset B(x,r)$ for every $n\in \mathbb{N}$. Hence, by Lemma \ref{ogr}, $S(A_n) \subset B(0,2r) \subset \bar{B}(0,2r)$ for all $n$. This closed ball is compact since $X$ is complete and locally compact. Therefore, for each $n\in \N$, we have $S(A_n) \in K(\bar{B}(0,2r))$ and this space is compact. Therefore, the sequence $(S(A_n))$ admits a subsequence convergent to some $D \in K(X)$. By Lemma \ref{conv}, $D \subset \{0\}$, thus $D=\{0\}$. Concluding, we have proved that for every sequence $(A_n)$, the corresponding sequence $(S(A_n))$ admits a subsequence converging to $S(A)$, which implies the continuity of $S$ at $A$.
\end{proof}
From Corollaries \ref{imp1} and \ref{imp2}, we immediately obtain the main theorem in this section.
\begin{theorem}
Let $(X,+,d)$ be a complete locally compact metric group with $d$ translation-invariant and such that every nonempty open ball in $X$ contains infinitely many elements. Let $A \in K(X)$. Then $S$ is continuous at $A$ if and only if $S(A) = \{0\}$.
\end{theorem}
\begin{remark}
It is easy to see that if there exists a non-empty open ball containing only finitely many elements, then the continuity of $S$ at some $A$ does not imply that $S(A) = \{0\}$ (it suffices to consider two-element subset of such a ball). On the other hand, it remains unclear whether the assumption of the completeness and the local compactness $X$ is necessary for the assertion of Corollary \ref{imp2} to hold.
\end{remark}


Although $S$ is not continuous at every set, it satisfies a weaker property of the upper-semicontinuity of $S$.

\begin{definition}[Upper-semicontinuity in $K(X)$]
We say that a function $T:K(X) \to K(X)$ is upper-semicontinuous at $A \in K(X)$ if for every $\ve >0$, there exists $\delta >0$ such that for all $B \in K(X)$, whenever $d_H(A,B) < \delta$, we have
\[
T(B) \subset T(A)_\ve:= \{x \in X\colon \exists_{a \in A} \,\, d(x,a) < \ve\}.\]
\end{definition}

\begin{theorem}\label{semi}
Let $(X,+,d)$ be a complete locally compact metric group with translation-invariant metric $d$. Then $S$ is upper-semicontinuous at every $A \in K(X)$.
\end{theorem}
\begin{proof}
Assume, for sake of contradiction, that there exist $A \in K(X)$, $\ve >0$, and a sequence $(A_n)$ in $K(X)$ converging to $A$ such that $S(A_n) \not \subset S(A)_\ve$.
Similarly to Corollary \ref{imp2} the sequence $(S(A_n))$ admits a subsequence $(S(A_{k_n}))$ converging to some $D \in K(X)$. Then, there exists $m \in \N$ such that for all $n \geq m$, $S(A_{k_n}) \subset D_\ve$. By Lemma \ref{conv}, we have
\[
D \subset S(A), \qquad (A_{k_n}) \subset D_\ve\subset S(A)_\ve,\]
yielding a contradiction. Therefore, $S$ is upper-semicontinuous at $A$.
\end{proof}

\section{Achievement sets, sets of P-sums and the spectre}

The main applications of the center of distances have been related to the study of achievement sets of the series, that is, the sets of all possible subsums of a given series. Specifically, if $\sum_n x_n$ is an absolutely convergent series in a Banach space $X$, the achievement set of this series is defined as
\[
E(x_n):=\left\{\sum_{n\in A} x_n \colon A \subset \N\right\}.\]
The following result serves as the main link between the center of distances and the achievement sets.
\begin{theorem} \cite{Bielas} \label{wyrazywcentrum}
Let $\sum_n x_n$ be an absolutely convergent series. For every $n \in \N$, we have $|x_n| \in C(E(x_n)).$
\end{theorem}
An analogous result holds for the spectre.\begin{theorem} \cite{Piotrus}
Let $\sum_n x_n$ be an absolutely convergent series. For any $n \in \N$, we have $x_n \in S(E(x_n)).$
\end{theorem}
Theorem \ref{wyrazywcentrum} was generalized in \cite[Lemma 2.5]{AB1}.
We now present the corresponding analogue for the spectre.
\begin{theorem}
Let $\sum_n x_n$ be an absolutely convergent series. Assume that
$$x_k=x_{k+1} = \ldots=x_{k+2j-2}.$$
Then $j x_k \in S(E(x_n))$.
\end{theorem}

\begin{proof}
Let $x \in E(x_n).$ Then, there exists a subset $A \subset \N$ such that $x = \sum_{n\in A}x_n$. Define $B:=\{k, k+1, \ldots, k+2j-2\}$. This set contains $2j-1$ elements, so either $|A \cap B| \geq j$ or $|(\N \setminus A) \cap B| \geq j$.
 In the first case, $x-jx_k \in E(x_n)$, and in the second, $x+jx_k \in E(x_n).$ Therefore, $x \in S(E(x_n))$.
\end{proof}

Denote by $\A$ the family of achievement sets in $K([0,1])$.

Let us denote the family of P-sums in $K([0,1])$ by $\mathcal{P}$, i.e.,
\[
\mathcal{P}=  \left\{T\subset [0,1]\colon \exists_{P \ \text{finite}} \exists_{(a_n)_{n\in\N} \subset\mathbb{R}} \,T = \{ \sum_{n\in \N} \xi_n a_n  \, : \, \xi_n \in P  \} \right\}.
\]

It is known that $\A\subset\p$ and that this inclusion is proper. We now exhibit another significant topological distinction between families $\A$ and $\p$ in the hyperspace of compact sets.
Before presenting the next result, we need a few concepts and facts from the theory of achievement sets that will be used in the course of the proof. If $\sum a_n$ is a convergent series of non-negative and non-increasing terms, then for $k\in \mathbb N_0$ we define the set of $k$-initial subsums by
$$
F_k\ =\ F_k(a_n)\ :=\ \ \left\{ x\in\mathbb R: \quad \exists_{A\subset\{1,2,\ldots,k\,\}}\quad x=\sum_{n\in A}a_n\,\right\}
$$
and the set of subsums of the $k$-th remainder of the series by
$$
E_k\ =\ E_k(a_n)\ :=\ \ \left\{ x\in\mathbb R: \quad \exists_{A\subset\{k+1,k+2,\ldots,\}}\quad x=\sum_{n\in A}a_n\,\right\}.
$$
Clearly,
$$
E\ =\ F_k\,+\,E_k\ \ =\ \ \bigcup_{f\in F_k}(f+E_k)
$$
for all $k\in\mathbb N_0$ \cite{BFPW1}.

Bounded components of $\mathbb R\setminus E$ are called gaps of $E$. An $E$-gap $(\alpha,\,\beta)$ is said to be dominating if all $E$-gaps lying to the left of it are shorter than $\beta-\alpha$. The Third Gap Lemma \cite[Lemma 2.4]{AB1} says that if $(\alpha,\,\beta)$ is a dominating $E$-gap then $\beta=a_m$ and $\alpha=\sum_{n>m}a_n$ for some $m\in\mathbb N$.

The following easy lemma will be also useful.
\begin{lemma}\label{rown}
Let $C,D,E,F \in K([0,1]).$ Then
$$d_H(C+D,\,E+F)\,\le\,d_H(C,\,E)+d_H(D,\,F).$$
\end{lemma}

We have the following fact.
\begin{theorem} \label{dom}
    The family $\A$ is closed in $K_0([0,1])$ and hence in $K([0,1])$.
\end{theorem}

\begin{proof}
Let $E^{(n)}\xrightarrow[n\to\infty]{d_H}\ S$ in the space $K([0,1])$ where, for each $n$, $E^{(n)}=E(a_i^{(n)})$ are achievement sets of convergent series $\sum_{i\in \mathbb N}a_i^{(n)}$ of non-decreasing and non-negative terms and of sums belonging to $[0,\,1]$. Since $0\in E^{(n)}$ for all $n$, it must also belong to $S$. Thus, if $S$ is a singleton, then $S$ must be equal to $\{0\}$, and therefore $S \in \A$.

From now on we assume that $S$ consists of at least two points. For a non-singleton  $A \in K([0,1])$ denote $\widehat{A} := \frac1{\max A} \cdot A$.

If $A_n\to S$ in $K([0,1])$, then the sets $A_n$ consist of at least two points for all sufficiently large $n$ and $\widehat{A_n} \to \widehat{S}$.
A set $B\subset\mathbb R$ is achievable if and only if $\alpha B$ is achievable for all $\alpha\in\mathbb R$. In particular, if $\widehat{S}$ is achievable, so is $S$. Thus, without loss of generality, we can assume that $S$ and all sets $E^{(n)}$ are closed subsets of $[0,\,1]$ with the minimal element equal to 0 and the maximal one equal to 1.

If $S$ has no gaps, then $S=[0,\,1]$ and is the achievement set of $\sum\frac1{2^n}$. Therefore, let us assume that our limit $S$ has at least one gap and at least two elements, as we previously assumed.

Consider any dominating $S$-gap $(\alpha,\,\beta)$. Then, since $E^{(n)}\xrightarrow{d_H}S$, we get  (removing finitely many initial terms of the sequence $(E^{(n)})_{n\in\mathbb N}$ if necessary) that $[\alpha,\,\beta]$ is the $d_H$-limit of closed intervals $[\alpha^{(n)},\,\beta^{(n)}]$ where each $(\alpha^{(n)},\,\beta^{(n)})$ is a dominating gap of $E^{(n)}$. Even more, we may assume and we do that $\sup_n\alpha^{(n)}<\inf_n\beta^{(n)}$.

By the Third Gap Lemma, for every $n$ there is an index $k_n$ such that $\beta^{(n)}=a^{(n)}_{k_n}$. Since $\sum_ia^{(n)}_i=1$ for all $n$, the sequence $(k_n)_{n\in\mathbb N}$ is bounded. Thus, passing to a subsequence if necessary, we have
$$
\exists_{k=k_{\alpha,\beta}\in\mathbb N}\,\forall_{ n\in \N}\,\, k_n=k.
$$
Further, passing to a subsequence if necessary,  we see that all sequences $(a^{(n)}_i)_{n\in\mathbb N}$ for $i=1,\,\ldots,\, k$ are convergent. Denote their limits by $a_i$, respectively. Clearly, $a_1\ge a_2\ge\ldots\ge a_k=\beta$. Then
\begin{equation}
\label{star}
F_k^{(n)}:=F_k((a^{(n)}_i)_{i=1}^k)\ \xrightarrow[n\to\infty]{d_H}\ F_k:=\, F_k((a_i)_{i=1}^k).
\end{equation}
For $n$ such that $d_H(E^{(n)},\,S)<\frac12(\beta-\alpha)$, we have
$$
d_H(E^{(n)},\,S)\ =\ \max\bigl\{ d_H(E^{(n)}_k,\,S\cap[0,\,\alpha]),\,d_H(E^{(n)}\cap[a_k^{(n)},\,1]),\,S\cap[\beta,\,1])\,\bigr\}
$$
and hence
$$
E^{(n)}_k\ \xrightarrow[n\to\infty]\ S\cap[0,\,\alpha]\qquad\text{and}\qquad E^{(n)}\cap[a^{(n)}_k,\,1]\ \xrightarrow[n\to\infty]\ S\cap[\beta,\,1].
$$
On the other hand,
$$
E^{(n)}\cap[a_k^{(n)},\,1]\ =\ F_k^{(n)}\setminus\{0\}\,+\, E^{(n)}_k
$$
and therefore
$$
S\cap[\beta,\,1] \ \overset{\eqref{star}}{=}\ F_k\setminus\{0\}\,+\,S\cap[0,\,\alpha]
$$
which implies that $S=F_k\, + \,S\cap[0,\,\alpha]$ which is our key observation.

Moreover, by Lemma \ref{rown}
\begin{equation}
\label{twostars}
d_H(F_k,\,S) = d_H(F_k+\{0\},F_{k}+S\cap[0,\alpha])\ \le\ d_H(\{0\},\,S\cap[0,\,\alpha])\ =\ \alpha.
\end{equation}

Since $S$ has at least one gap, exactly one of the following two cases holds:
\begin{itemize}
\item[(A)] $S$ has the most left dominating gap;
\item[(B)] $S$ does not have the most left dominating gap.
\end{itemize}

Consider the case (A) first. Let $(\alpha,\,\beta)$ be the most left dominating gap of $S$. By our earlier considerations there is a finite sequence $a_1\ge a_2\ge \ldots\ge a_k>0$ such that $S=F_k+S\cap[0,\alpha]$. Since there are no $S$-gaps in $[0,\alpha]$, $S \cap [0, \alpha]$ must be equal to $[0,\alpha]$. Then taking any series $\sum b_i$ such that $E(b_i)=[0,\,\alpha]$ (it could be $\alpha=0$, but it does not cause any problems), we get $S=E(\hat{a}_n)$ where $\hat{a}_i:=a_i$ for $i=1,\,\ldots,\,k$ and $\hat{a}_i:=b_{i-k}$ for $i>k$.

Finally, consider the case (B). Then $S$ has infinitely many dominating gaps. Observe that they cannot have common endpoints. Indeed, if $(\alpha,\,\beta)$ and $(\beta,\,\gamma)$ were two dominating  gaps of $S$, then
because
$$
S\ =\ F_{k_{\alpha,\beta}}\,+\,S\cap[0,\,\alpha],
$$
 we see that $0$ is an isolated point of $S\cap[0,\,\alpha]$ which means that $S$ has the most left dominating gap, a contradiction with (B).

 Put all dominating $S$-gaps in the natural decreasing order $\bigl((\alpha_i,\,\beta_i)\bigr)_{i=1}^\infty$, that is, $\beta_{i+1}<\alpha_i$ for all $i$. Then using our key observation and induction we construct a sequence $(a_i)$ of positive numbers and a sequence $(k_i)$ of positive integers such that for each $n$
 $$ S= F_{k_n}(a_i)+S \cap [0,\alpha_n].$$
Then, by \eqref{twostars}, $d_H(F_{k_n}(a_i),\,S)=\alpha_n$. Thus, $F_{k_n}(a_i)\to S$ and simultaneously $F_{k_n}(a_i)\to E(a_i)$ which yields $S=E(a_i)$.
\end{proof}

Observe that, with only minor modifications, the same proof shows that the family of all achievable subsets of $\mathbb R$ is closed in $(K(\mathbb R),\, d_H)$.

For $P$-sums, however, the situation is different.
The following result was established in an even more general form in \cite{AC}.
\begin{proposition}\label{podob}
    Let $A$ be the set of $P$-sums for a finite set $P$ of nonnegative elements containing $0$. Let $(a,b)$ be a gap in $A$. Then, there exists $\ve >0$ such that \[b+(A\cap [0,\ve])=A\cap [b,b+\ve].\]
\end{proposition}
\begin{theorem} \label{przedzialy} \cite[Thm. 8.1]{GM}
    Every union of a finite family of closed and bounded intervals in $[0,\,\infty)$ is a set of $P$-sums.
\end{theorem}
We can now prove the following theorem.

\begin{theorem}
    The family $\p$ is not closed in $K([0,1]).$
\end{theorem}

\begin{proof}
    Let $C, D\subset \mathbb{R}$ be Cantor sets whose convex hull is $[0,\frac14]$, and such that $C\cap D$ is countable. Let $A = C \cup D+\frac12$. Then $A$ is a Cantor set. Moreover, $(\frac14,\frac12)$ is a gap in $A$. However,
\[
\frac12+([0,\ve] \cap A)\neq [\frac12,\frac12+\ve] \cap A
\]
for any $\ve$. Therefore, by Proposition \ref{podob}, $A$ is not a set of $P$-sums (if such a $P$ existed, it would have to contain $0$ and only nonnegative elements). However, we can obtain $A$ as a limit of sequence of finite unions of closed intervals, which are sets of $P$-sums, by Theorem \ref{przedzialy}. Thus, the family $\p$ is not closed in $K([0,1])$.
\end{proof}

It is worth noting that, since the family of $P$-sums contains all finite sets, then $\p$ is dense in $K([0,1]).$

In the paper \cite{Bartur} it was proved that $\A$ is meager in $K_0[0,1]$. From this fact and Theorem \ref{dom} we obtain.
\begin{theorem}
    The family of achievement sets is nowhere dense in $K_0([0,1]).$
\end{theorem}

We conclude this part of the paper with a deeper question concerning the topological complexity of the family $\p$.

\begin{problem}
    What is the Baire category of the family $\p$ in $K([0,1])$ (or $K_0([0,1])$)?
\end{problem}

\section{Spectre and achievement sets in the plane}

At the XXXVII International Summer Conference on Real Functions Theory (Rowy, 2023), Jacek Marchwicki posed a question whether is it possible to prove analogues of three gap lemmas -- known for achievement sets on the real line -- for  achievement sets in the plane.
The answer is positive, although it requires introducing some additional assumptions on the sequence. On the real line similar assumptions were natural, but in the plane they are, unfortunately, rather restrictive.
Moreover, we need to introduce a notion of gap in $\R^2$:
\begin{itemize}
\item A set $(a,b)\times(c,d)$ is called a rectangular gap of a set $A\subset \R^2$ if $[a,b]\times[c,d]\cap A = \{(a,c),(b,d)\}$.
\item An interval $(a,b)$ is called an $x$-gap of $A$ if there exist $c,d\in \R$ such that $(a,c)$ and $(b,d)\in A$, but $((a,b)\times \R) \cap A=\emptyset$.
\item An interval $(c,d)$ is called a $y$-gap of $A$ if there exist $a,b\in \R$ such that $(a,c)$ and $(b,d)\in A$, but $(\R\times (c,d)) \cap A=\emptyset.$
\end{itemize}
The following result is a two-dimensional counterpart of the First Gap Lemma.
\begin{lemma}
Let $((x_n,y_n))_{n\in\N}$ be a sequence in $[0,\infty)^2$ such that the series $\sum_n(x_n,y_n)$ converges, and let $k\in \N$.
\begin{itemize}
    \item[a)] Let $A := \{n:x_n < x_k\}$. If $x_k>\sum_{n\in A} x_n$, then $(\sum_{n\in A}x_n,x_k)$ is an $x$-gap of $E(x_n,y_n)$.
\item[b)]  Let $A := \{n:y_n < y_k\}$. If $y_k>\sum_{n\in A}y_n$, then $(\sum_{n\in A}y_n,y_k)$ is a $y$-gap of $E(x_n,y_n)$.
\item[c)] Assume that $A:=\{n:x_n < x_k\}=\{n:y_n < y_k\}$. If $x_k>\sum_{n\in A}x_n$ and $y_k>\sum_{n\in A}y_n$, then the set $(\sum_{n\in A}x_n,x_k)\times (\sum_{n\in A} y_n,y_k)$ is a rectangular gap of $E(x_n,y_n)$.
\end{itemize}
\end{lemma}
\begin{proof}$\;$\linebreak\vspace{-0.35cm}
\begin{itemize}
    \item[Ad a)] We have $(\sum_{n\in A}x_n,\sum_{n\in A}y_n) \in E(x_n,y_n)$ and $(x_k,y_k) \in E(x_n,y_n)$. Assume that there exists point $(a,b) \in E(x_n,y_n)$ such that $a \in (\sum_{n \in A} x_n,x_k)$. Then there exists a set $B\subset \N$ such that $\sum_{n\in B} x_n = a$. If $B\subset A$, then $a<\sum_{n\in A}x_n$, a contradiction. If there exists $m \in B \setminus A$, then $x_m \geq x_k$, so $a\geq x_m\geq x_k$, and again we obtain a contradiction. Thus, $(\sum_{n\in A}x_n, x_k)$ is an $x$-gap of $E(x_n,y_n)$.

   \item[Ad b)] The argument is analogous to that of a).

    \item[Ad c)] The claim follows immediately from a) and b).
\end{itemize}
\end{proof}
For an absolutely convergent series $\sum_n x_n$ in $\R^k$ for $k\in\N$, we define (similarly as on the line) the sets \[
F_n:=\{\sum_{i\in A}x_i\colon A\subset \{1,2,\ldots,n\}\},\]
and
\[
E_n:=\{\sum_{i\in A}x_i\colon A\subset \{n+1,n+2,\ldots\}\}.
\]
The next result can be regarded as the two-dimensional analogue of the Second Gap Lemma in $\R^2.$

\begin{lemma}
   Let $(x_n,y_n)$ be a sequence in $[0,\infty)^2$ such that the series $\sum_n(x_n,y_n)$ converges. Assume that the set $(a,b)\times(c,d)$ is a rectangular gap of $E(x_n,y_n)$. Let
   \[
   k:=\max\{n\in\N\colon x_n\geq b-a \vee y_n \geq d-c\}.
   \]
   Then $(b,d) \in F_k$. Moreover, there exists $f\in F_k$ such that $(a,c) = f+\sum_{n>k}(x_n,y_n)$.
\end{lemma}
\begin{proof}
    Suppose that $(b,d)\notin F_k$. Since $(b,d) \in E(x_n,y_n)$ there is $A\subset \N$ such that $(b,d) = \sum_{n\in A} (x_n,y_n)$. Since $(b,d)\notin F_k$, there exists $m \in A$ satisfying $m>k$. However, $(b-x_m,d-y_m) \in E(x_n,y_n)$, with $b-x_m > a$ and $d-y_m > c$. At the same time $(a,b)\times(c,d)$ is a rectangular gap of $E(x_n,y_n)$, which yields a desired contradiction. Thus, $(b,d)\in F_k$.

By the definition of a rectangular gap, $(a,c) \in E(x_n,y_n)$. Hence, there exists $B\subset \N$ such that $(a,c) = \sum_{n\in B} (x_n,y_n)$. Let $f=\sum_{n\in B, n\leq k} (x_n,y_n).$ Obviously, $f \in F_k$. We now show that $\{k+1,k+2,\ldots\} \subset B$. Assume that for some $l > k,$ $l\notin B.$ Then $x_l < b-a$ and $y_l < d-c$. Moreover, \[(a,c)+(x_l,y_l) \in E(x_n,y_n) \cap (a,b)\times(c,d),\] a contradiction. Thus, $(a,c) = f+\sum_{n>k}(x_n,y_n).$

\end{proof}

The following example demonstrates that the Third Gap Lemma does not extend to the plane.
\begin{example}
    Let \[(x_1,y_1)=\left(\frac78,\frac18\right), \,(x_2,y_2)=\left(\frac18,\frac78\right), \,(x_3,y_3)=(x_4,y_4)=\left(\frac{3}{16},\frac{3}{16}\right)\]
    and put $(x_n,y_n)=(0,0)$ for $n>4$. Then
    \begin{eqnarray*}
    E(x_n,y_n)&=& \Bigg\{(0,0),(1,1),
    \left(\frac{3}{16},\frac{3}{16}\right),
    \left(\frac{3}8,\frac38\right),
    \left(\frac78,\frac18\right),
    \left(\frac{17}{16},\frac{5}{16}\right),
    \\
    & \, & \left(\frac{5}{4},\frac12\right),
    \left(\frac18,\frac78\right),
    \left(\frac{5}{16},\frac{17}{16}\right),
    \left(\frac12,\frac{5}{4}\right),
    \left(\frac{19}{16},\frac{19}{16}\right),
    \left(\frac{11}{8},\frac{11}{8}\right)
    \Bigg\}.
    \end{eqnarray*}
    In particular, the set $(\frac38,1)\times(\frac38,1)$ is the largest rectangular gap of $E(x_n,y_n).$ However, $(1,1)\neq  (x_n,y_n)$ for any $n\in\N.$
\end{example}

The next result serves as a counterpart of Fact 2.5 and Fact 2.8 from \cite{BBP}, formulated in $\R^k$, with the spectre replacing the center of distances.

\begin{proposition}
    Let $(x_i)$ be a sequence in $\R^k$, where $k\in\N,$ such that the series $\sum_i x_i$ is absolutely convergent. Let $n\in \N$. Then
    \begin{itemize}
        \item[(1)] $S(F_n)\subset S(F_{n+1}) \subset S(E(x_i))$,

\item[(2)] $S(E_{n+1})\subset S(E_{n}) \subset S(E(x_i))$,

    \end{itemize}
\end{proposition}
\begin{proof}$\;$\linebreak\vspace{-0.35cm}
\begin{itemize}
    \item[Ad (1)] Let $y\in S(F_n)$ and $x\in F_{n+1}$. Then $x=f+\ve x_{n+1}$, where $v\in \{0,1\}$ and $f\in F_n$. Since $y\in S(F_n)$, we have $f+y \in F_n$ or $f-y\in F_n$, and consequently $x+y=f+\ve x_{n+1}+y \in F_{n+1}$ or $x-y=f+\ve x_{n+1}-y \in F_{n+1}$. Hence $y \in S(F_{n+1})$.

    Now let $z\in E(x_i)$. Then $z=g+h$, where $g\in F_n$ and $h\in E_n$. Since $y\in S(F_n)$, we obtain that $g+y \in F_n$ or $g-y\in F_n$, and thus $z+y=g+h+y \in E(x_i)$ or $z-y=g+h-y \in E(x_i)$. In either case, it follows that $y \in S(E(x_i))$.

   \item[Ad (2)] It suffices to prove the first inclusion, since $E(x_i)=E_0$. By Remark \ref{podst} (3) and Proposition \ref{capcup}, we obtain
    \begin{eqnarray*}
    S(E_{n+1})&=&S(E_{n+1})\cap S(E_{n+1})
    = S(E_{n+1}) \cap S(x_{n+1}+E_{n+1})\\
    &=& S(E_{n+1} \cup (x_{n+1}+E_{n+1}))=S(E_n).
    \end{eqnarray*}
\end{itemize}
\end{proof}

We conclude by noting that, while several one-dimensional gap results admit natural planar counterparts, the full analogy breaks down beyond the second lemma, illustrating the importance of the structural differences between achievement sets in $\mathbb{R}$ and $\mathbb{R}^2$.

\section*{Acknowledgments}

Filip Turoboś wishes to express his deepest gratitude to his wife, Nicole Meisner, and to his friends — including the coauthors of this paper — for their unwavering love, prayers, and support throughout his battle with cancer. Their presence, shared time, and encouragement have been an invaluable source of strength and hope, enabling him to emerge victorious from this unfair fight.


\begin{thebibliography}{abcd}
\bibitem{AC} R. Anisca, C. Chlebovec. \emph{On the structure of arithmetic
sum of Cantor sets with constant ratios of dissection}. Nonlinearity \textbf{%
22} (2009), 2127--2140.

\bibitem{BBP} M. Banakiewicz, A. Bartoszewicz, F. Prus-Wiśniowski. \emph{The center of distances of some multigeometric series}. arXiv:1907.03800.

\bibitem{Bartur} A. Bartoszewicz, M. Filipczak, G. Horbaczewska, S. Lindner, F. Prus-Wiśniowski, \emph{On the operator of center of distances between the spaces of closed subsets of the real line}, Topol. Methods Nonlinear Anal. \textbf{63}(2) (2024), 413--427.
\bibitem{BFPW1} A. Bartoszewicz, M. Filipczak, F. Prus-Wi\'sniowski, \textit{Topological and algebraic aspects of subsums of series}, Traditional and present-day topics in real analysis, 345--366, Faculty of Mathematics and Computer Science. University of \L \'od\'z, \L \'od\'z, 2013.

\bibitem{AB1} A. Bartoszewicz, S. Głąb and J. Marchwicki, \emph{Recovering purely atomic finite measure
from its range}, J. Math. Anal. Appl. \textbf{467}(2) (2018), 825--841.
\bibitem{ABart2} A. Bartoszewicz, S. Głąb, M. Filipczak, F. Prus-Wiśniowski, J. Swaczyna,
\emph{On generating regular Cantorvals connected with geometric Cantor sets}, Chaos Solit. Fractals \textbf{114}, (2018), 468--473.
\bibitem{Bielas} W. Bielas, S. Plewik, M. Walczyńska, \emph{On the center of distances}, Eur. J. Math. \textbf{4} (2018), 687--698.
\bibitem{hausdorfik} T. Birsan, D. Tiba, \emph{One hundred years since the introduction of the set distance by Dimitrie Pompeiu}, in: Ceragioli, F., Dontchev, A., Futura, H., Marti, K., Pandolfi, L. (eds) System Modeling and Optimization. CSMO 2005. IFIP International Federation for Information Processing, vol 199. Springer, Boston, MA.
\bibitem{GM} S. Głąb, J. Marchwicki, \emph{On arithmetic sums of Cantor sets: P-sums vs. achievement sets}, to appear in: Riv. Math. Univ. Parma.
\bibitem{Ku} M. Kula, \emph{Center of distances and Bernstein sets},  Real Anal. Exchange {\bf{50}}(1) (2025), 207--212.
\bibitem{Piotrus} M. Kula, P. Nowakowski, \emph{Achievement sets of series in $\mathbb{R}^2$}, Results Math. \textbf{79} (2024), art. no. 221.
\bibitem{S}
 S.M. Srivastava, \textit{A course on Borel Sets}, Graduate Texts in Mathematics, vol. 180, Springer Verlag, New York, Berlin, Heidelberg 19

\end{thebibliography}
\end{document}